\documentclass[reqno]{amsart}
\usepackage{amsmath,amssymb, amsthm}
\usepackage{hyperref}
 \newtheorem{theorem}{Theorem}[section]
\newtheorem{proposition}{Proposition}[section]
 
\newtheorem{corollary}{Corollary}[section]
\newtheorem{lemma}{Lemma}[section]
 \newtheorem{remark}{Remarks}[section]
\usepackage{hyperref}
\def\eps{\varepsilon}
\usepackage{color}

\begin{document}
\title{Liouville-type theorems and bounds of \\ solutions for Hardy-H\'enon elliptic systems} 
\author{Quoc Hung PHAN} 

\address{Universit\'e Paris 13, CNRS UMR 7539,
Laboratoire Analyse, G\'eom\'etrie et Applications 93430 Villetaneuse, France}

\email{phanqh@math.univ-paris13.fr}

\keywords{Hardy-H\'enon system; Nonexistence; Liouville-type theorem}
\subjclass{primary 35J60, 35B33; secondary 35B45}
\begin{abstract}
We consider the Hardy-H\'enon system
$-\Delta u =|x|^a v^p$, $-\Delta v =|x|^b u^q$  with $p,q>0$ and $a,b\in {\mathbb R}$
and we are concerned in particular with the Liouville property,
i.e. the nonexistence of positive solutions in the whole space ${\mathbb R}^N$.
In view of known results, it is a natural conjecture that this property should be true if and only if $(N+a)/(p+1)+$ $(N+b)/(q+1)>N-2$. 
In this paper, we prove the conjecture for dimension $N=3$ in the case of bounded solutions
and in dimensions $N\le 4$ when $a,b\le 0$, among other partial nonexistence results. 
As far as we know, this is the first optimal Liouville type result for the Hardy-H\'enon system.
Next, as applications, we give results on singularity and decay estimates as well as a priori bounds of positive solutions.
\end{abstract}
\maketitle
\section{Introduction}
We study the semilinear elliptic systems  of Hardy-H\'enon type
\begin{align}\label{1}
\begin{cases}
-\Delta u=|x|^av^p,\quad x\in \Omega,\\
-\Delta v=|x|^bu^q,\quad x\in \Omega,
\end{cases}
\end{align}
where $p,q >0$,  $a,b\in {\mathbb R}$ and $\Omega$ is a domain of ${\mathbb R}^N$, $N\geq 3$.

Throughout this paper, unless otherwise specified, solutions are considered in the class
\begin{equation}\label{solclass}
C^2(\Omega\setminus\{0\})\cap C(\Omega).
\end{equation}

Let us first note that, if $\min\{a,b\}\leq -2$, 
then (\ref{1}) has no positive solution in class (\ref{solclass}) in any domain $\Omega$ containing the origin \cite[Proposition 2.1]{BV99}. We therefore restrict ourselves to the case $\min\{a,b\}> -2$.

We are interested in the Liouville type theorem- i.e. the nonexistence of positive solution in the entire space $\Omega={\mathbb R}^N$- and its applications such as a priori bounds 
 and singularity and decay estimates of solutions. 

We recall the case $a=b=0$ of (\ref{1}), the so-called Lane-Emden system, which  has been widely studied by many authors. Here, the Lane-Emden conjecture states that there is no positive classical solution in $\Omega={\mathbb R}^N$ if and only if 
\begin{align}\label{hyper}
\frac{N}{p+1}+\frac{N}{q+1}>N-2.
\end{align}
This conjecture is known to be true for radial solutions in all dimensions \cite{Mit96}. For non-radial solutions, in dimension $N\leq 2$, the conjecture is a consequence of a result of Mitidieri and Pohozaev \cite{MP01}. In dimension $N=3$, it was proved by Serrin and Zou \cite{SZ96} under the additional assumption that $(u,v)$ has at most polynomial growth at $\infty$. This assumption was then removed by Polacik, Quittner and Souplet \cite{PQS07} and hence the conjecture is true for $N=3$. Recently, the conjecture was proved for $N=4$ by Souplet \cite{Sou09}, and some  partial results were also established for $N\geq 5$ (see \cite{Sou09, BM02, Lin98, CL09}).

For the general system with $a\ne 0$ or $b\ne 0$, the Liouville property is less understood. In fact, the nonexistence of {\bf supersolution} has been studied in \cite{AS11, MP01}. 
The following result is essentially known.

\medskip


\par\noindent {\bf Theorem A.} 
{\it Let $a,b>-2$ and $N\geq 3$. If $pq\leq 1$, or if $pq>1$ and 
\begin{equation}\label{condSupersol}
\max\biggl\{\frac{2(p+1)+a+bp}{pq-1}, \frac{2(q+1)+b+aq}{pq-1} \biggr\}\geq N-2,
\end{equation}
 then system (\ref{1}) has no positive supersolution in $\Omega={\mathbb R}^N$.
}

\medskip

Moreover, it is not difficult to check that condition (\ref{condSupersol}) is optimal for supersolutions 
(consider functions of the form $u(x)=(c_1+c_2|x|^2)^{-\gamma_1}$, $v(x)=(c_3+c_4|x|^2)^{-\gamma_2}$).
Miditieri and Pohozaev 
proved Theorem A for $p, q\geq 1$ by rescaled test-function method (see \cite[Section~18]{MP01}).
Theorem~A for all $p,q>0$ can be proved by an 
argument totally similar to that of Serrin and Zou in \cite{SZ96}.
There,  the authors treated  
the special case $a=b=0$, but this argument still works for the general case $a,b>-2$. 
However, 
their proof is rather involved, especially for  $p<1$ or $q<1$. 
Very recently, Amstrong and Sirakov \cite{AS11} developed 
a new maximum principle type argument which,
among other things, allows for a simpler proof of Theorem~A for all $p,q>0$. 
It follows from the arguments in \cite[Section~6]{AS11}. 
Also, Theorem~A remains true if $\Omega$ is an exterior domain.
\smallskip

As usual, it is expected that the optimal range of nonexistence for solutions
should be larger than for supersolutions. However, this question seems still difficult, even in the 
special case $a=b=0$.
Furthermore, 
even for the scalar equation $-\Delta u=|x|^au^p$, the optimal condition for 
nonexistence of positive solution on the whole of ${\mathbb R}^N$
has not been completely settled yet when $a>0$ 
(see the recent paper \cite{PhS} and cf.~\cite{GS81a, BVG10}).
Concerning system (\ref{1}), the following optimal result regarding radial solutions 
is known \cite{BVG10}.

\medskip
\par\noindent {\bf Proposition B.}
{\it Let $a,b>-2$. Then system (\ref{1}) has no positive radial solution in $\Omega={\mathbb R}^N$ if and only if 
\begin{align}\label{1a}
\frac{N+a}{p+1}+\frac{N+b}{q+1}>N-2.
\end{align}
}

The hyperbola
\begin{align}\label{2a}
\frac{N+a}{p+1}+\frac{N+b}{q+1}=N-2.
\end{align}
thus plays a critical role in the radial case and this, combined with the case of Lane-Emden system, leads to the following conjecture.

\medskip
\par\noindent {\bf Conjecture C.}
{\it Let $a,b>-2$. Then system (\ref{1}) has no positive solution in $\Omega={\mathbb R}^N$ if and only if 
$(p,q)$ satisfies (\ref{1a}).
}
\medskip

In this paper, we prove the conjecture for dimension $N=3$ in the class of bounded solutions, 
and for dimensions $N\le 4$ when $a,b\le 0$, without any growth assumption
(among other partial results in higher dimensions).

\medskip

{\bf From now on, we restrict ourselves to the case $pq>1$ (cf. Theorem A), and without loss of generality, it will be assumed that}
$$p\geq q.$$
Let us denote 
 \begin{align}
\alpha =\frac{2(p+1)}{pq-1},\qquad 
 \beta=\frac{2(q+1)}{pq-1}.
\end{align}
Then (\ref{1a}) is equivalent to
\begin{align}\label{la1}
\alpha \Bigl(1+\frac{b}{2}\Bigr) + \beta\Bigl(1+\frac{a}{2}\Bigr)>N-2. 
\end{align}

 We have obtained the following Liouville type results.

\begin{theorem}\label{cond}
Let $a,b >-2$ and $N\geq 3$. Assume $pq>1$ and (\ref{1a}). 
If $N\geq 4$,  assume in addition that  
 \begin{align}
&0\leq a-b \leq (N-2)(p-q),\label{40}\\
&\alpha>N-3.\label{50}
\end{align}
Then system (\ref{1}) has no positive bounded solution in $\Omega={\mathbb R}^N$.
\end{theorem}

\begin{theorem}\label{cond2}
Let $a,b >-2$ and $N\geq 3$. 
Assume $pq>1$, $p\ge q$, (\ref{hyper}), (\ref{1a}) and (\ref{50}). 
 Then system (\ref{1}) has no positive solution in $\Omega={\mathbb R}^N$.
\end{theorem}

As an immediate consequence, we obtain Conjecture~C in the following special cases.

\begin{corollary}
\smallskip

(i) If $N=3$, then Conjecture~C is true for bounded solutions.~\footnote{After the completion of the present work, we received a preprint by M.Fazly and N.Ghoussoub
where they obtain Theorem \ref{cond} in the special case $N=3$ with $a,b\ge 0$.
They also prove interesting result about {\it solutions with finite Morse index } in the scalar case.}
\smallskip

(ii) If $N=3$ or $4$ and $a, b\le 0$, then Conjecture~C is true.

\end{corollary}

 \begin{remark}

(a) The proof of Theorem \ref{cond} 
uses the technique introduced by Serrin and Zou in \cite{SZ96} and further developed by Souplet in \cite{Sou09}, 
which is based on a combination of Rellich-Pohozaev identity, a comparison property between components via the maximum principle, Sobolev and interpolation inequality on $S^{N-1}$ and feedback and measure arguments.
As for the idea of the proof of Theorem \ref{cond2}, see after Theorem~\ref{th3} below.

\smallskip 

(b) Theorem \ref{cond} is still true for polynomially bounded solutions, i.e. if $u(x)\leq C|x|^q$ for $x$ large,
with some $q>0$. This follows from easy modifications of the proof.
Let us recall that Liouville type theorems for bounded solutions are usually sufficient for applications such as a priori estimates
and universal bounds, obtained by rescaling arguments (see \cite{GS81b, PQS07}).

\smallskip

(c) If $a+b\leq 2(4-N)/(N-3)$ then condition (\ref{50}) is a consequence of (\ref{1a}). Then it follows from Theorem~\ref{cond2} that Conjecture~C is true if $a,b\leq 0$ and $a+b\leq 2(4-N)/(N-3)$. 

\end{remark}

\vskip 0.2cm

We next study the strongly related question of singularity and decay estimates for solutions of system (\ref{1}). We have the following theorem.

\begin{theorem}\label{th3} Let $a,b >-2$ and $N\geq 3$. 
Assume $pq>1$, $p\ge q$, 
(\ref{hyper}) and (\ref{50}). Then there exists a constant $C=C(N,p,q,a,b)>0$ such that the following holds.
\smallskip 

(i) Any positive solution of system~(\ref{1}) in $\Omega=\{x\in{\mathbb R}^N;\, 0<|x|<\rho\}$ ($\rho>0$) satisfies
\begin{align}
u(x)\leq C|x|^{-\alpha -\frac{a+bp}{pq-1}}, 
\quad
v(x)\leq C|x|^{-\beta -\frac{b+aq}{pq-1}}   ,
\quad 0<|x|<\rho/2.\label{sing1}
\end{align}

(ii) Any positive solution of system~(\ref{1}) in $\Omega=\{x\in{\mathbb R}^N;\, |x|>\rho\}$ ($\rho\ge 0$) satisfies
\begin{align}
u(x)\leq C|x|^{-\alpha -\frac{a+bp}{pq-1}},
\quad v(x)\leq C|x|^{-\beta -\frac{b+aq}{pq-1}}
,
\quad  |x|>2\rho.\label{sing3}
\end{align}
\end{theorem}
 The proof of Theorem~\ref{th3} is based on:  
 
-{\hskip 1mm}a change of variable, that allows to replace the coefficients $|x|^a, |x|^b$ with smooth functions which are bounded and bounded away from $0$ in a suitable spatial domain;

-{\hskip 1mm}a generalization of a doubling-rescaling argument from \cite{PQS07} (see Lemma~\ref{lem3} below);

-{\hskip 1mm}a known Liouville theorem for the Lane-Emden system \cite{Sou09}.

With Theorem~\ref{th3} at hand (along with the corresponding decay estimates for the gradients~--~cf.~Proposition \ref{th3B} below),
one can then deduce Theorem~\ref{cond2} from the Rellich-Pohozaev identity.

\vskip 0.2cm
Finally, as an application of our Liouville theorems, we derive a priori bounds of solutions of  the  following boundary value problem 
\begin{align}\label{bound}
\begin{cases}
-\Delta u=|x|^av^p,\quad x\in \Omega, \\ 
-\Delta v=|x|^bu^q,\quad x\in \Omega, \\ 
(u,v)=(\varphi, \psi),\quad x\in \partial \Omega,
\end{cases}
\end{align}
where $\Omega\subset{\mathbb R}^N$ is a smooth bounded domain containing the origin, $\varphi, \psi \in C(\partial\Omega)$ are nonnegative. For this, we essentially follow the classical blow-up method of Gidas and Spruck~\cite{GS81b}. We have the following.

\begin{theorem}\label{th4}
Let $\varphi, \psi$ be nonegative functions in $C(\partial \Omega)$.
Under the assumptions of Theorem \ref{cond2}, all positive solutions of (\ref{bound}) in $C^2(\Omega\setminus \{0\})\cap C(\overline{\Omega})$  are uniformly bounded. 
\end{theorem}

  The boundary value problem (\ref{bound}) has been investigated, especially for the case $\varphi=\psi=0$, and the existence and non-existence of positive solutions have been established \cite{CR10,FPR08}. More precisely,
 the existence of a positive solution is obtained via   
variational methods and non-existence of nontrivial solutions in starshaped domains is 
 a consequence of a generalized Pohozaev-type identity. Further results on the asymptotic behavior of solutions for H\'enon systems with nearly critical exponent can be found in \cite{HY08}.

 \begin{remark}
The conclusions of Theorem~\ref{th3} remain true under the assumption that system (\ref{1}) does not admit positive bounded solution in $\Omega={\mathbb R}^N$ 
(instead of (\ref{hyper}) and (\ref{50})). 
As for Theorem~\ref{th4}, it remains true under the assumptions that both system (\ref{1}) and system (\ref{1}) with $a=b=0$ do not admit positive bounded solution in $\Omega={\mathbb R}^N$
(instead of (\ref{hyper}), (\ref{1a}) and (\ref{50})). 

\end{remark}

The rest of paper is organized as follows. In Section 2, we recall some functional inequalities, Rellich-Pohozaev identity and prove a comparison property between the two components. Section 3 is devoted to the proof of Theorem \ref{cond}. 
The proof is quite long and involved, and for the sake of clarity, we separate it in two cases: $N\geq4$ and $N=3$. Section~4 is devoted to applications of Liouville property, we establish the singularity and decay estimates as well as a priori bound of solutions. 
The proof of Theorem \ref{cond2} is then given in Section 5. Finally, for completeness, we collect in Appendix the proofs of some results which are more or less known.

\vskip 0.2cm
\section{Preliminaries}
For $R>0$, we set $B_R=\{x\in {\mathbb R}^N;\ |x|<R\}$.
We shall use spherical coordinates $r=|x|$, $\theta=x/|x| \in S^{N-1}$ and write $u=u(r,\theta)$. 
The surface measures on $S^{N-1}$ and on the sphere
$\{x\in {\mathbb R}^N;\ |x|=R\}$, $R>0$, will be denoted respectively by $d\theta$ and by $d\sigma_R$.
For given function $w=w(\theta)$ on $S^{N-1}$ and $1\leq k\leq\infty$, we set
$\|w\|_k=\|w\|_{L^k(S^{N-1})}$. 
When no confusion is likely, we shall denote
$\|u\|_k=\|u(r,\cdot)\|_k$.

\subsection{Some functional inequalities}

\begin{lemma}[Sobolev inequalities on $S^{N-1}$]\label{1b}
Let $N\geq 2, j\geq 1$ is integer and $1<k<\lambda \leq\infty$, $k\ne (N-1)/j$. For $w=w(\theta)\in W^{j,k}(S^{N-1})$, we have
\begin{align*}
\|w\|_\lambda \leq C(\|D^j_\theta w\|_k+\|w\|_1)
\end{align*}
where 
\begin{align*}
\begin{cases}
&\frac{1}{k}-\frac{1}{\lambda}=\frac{j}{N-1}, \text{ if } k<(N-1)/j,\\
&\lambda=\infty \quad\quad\quad\, \text{ if } k>(N-1)/j.
 \end{cases}
\end{align*}
and $C=C(j,k,N)>0$.
 \end{lemma}
See e.g \cite{SZ96}.
\begin{lemma}[Elliptic $L^p$- estimates on an annulus]\label{lem2a}\label{lemell}
Let $1<k<\infty$. For $R>0$ and $z=z(x)\in W^{2,1}(B_{2R}\setminus B_{R/4})$, we have
\begin{align*}
\int\limits_{B_R\setminus B_{R/2}}|D^2_xz|^kdx\leq C\bigg(\int\limits_{B_{2R}\setminus B_{R/4}}|\Delta z|^kdx+R^{-2k} \int\limits_{B_{2R}\setminus B_{R/4}}|z|^kdx\bigg),
\end{align*}
with $C=C(n,k)>0$.
\end{lemma}

\begin{lemma}[An interpolation inequality on an annulus]\label{lem3a}
For $R>0$ and $z=z(x)\in W^{2,1}(B_{2R}\setminus B_{R/4})$, we have
\begin{align*}
\int\limits_{B_R\setminus B_{R/2}}|D_xz|dx\leq CR\int\limits_{B_{2R}\setminus B_{R/4}}|\Delta z|dx+ CR^{-1}\int\limits_{B_{2R}\setminus B_{R/4}}|z|dx,
\end{align*}
with $C=C(n)>0$.
\end{lemma}

Lemmas \ref{lem2a} and \ref{lem3a} follow from the case $R = 1$ and an obvious dilation argument.
Lemma~\ref{lem2a} with $R = 1$ is just the standard elliptic estimate. For Lemma \ref{lem3a} with $R = 1$, 
see e.g.~\cite{PhS}.

\subsection{Basic estimates, identities and comparison properties}

We have the following basic integral estimates for solutions of (\ref{1}). 

\begin{lemma}\label{lem4a}
Let $pq>1$, $a,b>-2$, $N\geq 3$ and $(u,v)$ be a positive solution of (\ref{1}) in $\Omega={\mathbb R}^N$. Then there holds 
\begin{align}\label{01}
&\int\limits_{B_R\setminus B_{R/2}}|x|^av^p\,dx\le C R^{N-2-\alpha-\frac{a+bp}{pq-1}}, \;\;\; 
   \int\limits_{B_R\setminus B_{R/2}}|x|^bu^qdx\le C R^{N-2-\beta-\frac{b+aq}{pq-1}},\quad R>0,
\end{align} 
with $C=C(N,p,q,a,b)>0$.
\end{lemma}

A simple proof of Lemma \ref{lem4a} is given in appendix, based on ideas from \cite{AS11}.

\medskip
From  Lemmas \ref{lem2a}--\ref{lem4a} and H\"older's inequality, we easily deduce the following lemma.
\begin{lemma}\label{lem}
Let $pq>1$ $a,b>-2$, $N\geq 3$ and $(u,v)$ solution of (\ref{1}), there hold 

\begin{align*}
&\int\limits_{B_R\setminus B_{R/2}}udx\le C R^{N-\alpha-\frac{a+bp}{pq-1}}, \qquad 
\int\limits_{B_R\setminus B_{R/2}}vdx\le C R^{N-\beta-\frac{b+aq}{pq-1}},\\
&\int\limits_{B_R\setminus B_{R/2}}|D_xu|dx\le C R^{N-1-\alpha-\frac{a+bp}{pq-1}}, \qquad
\int\limits_{B_R\setminus B_{R/2}}|D_xv|dx\le C R^{N-1-\beta-\frac{b+aq}{pq-1}},\\
&\int\limits_{B_R\setminus B_{R/2}}|\Delta u|dx\le C R^{N-2-\alpha-\frac{a+bp}{pq-1}} , \qquad
   \int\limits_{B_R\setminus B_{R/2}}|\Delta v|dx\le C R^{N-2-\beta-\frac{b+aq}{pq-1}},
\end{align*}
with $C=C(N,p,q,a,b)>0$.
\end{lemma}
\medskip

The following Rellich-Pohozaev identity plays a key role in the proof of Theorem \ref{cond}. It is probably known 
(see e.g \cite{CR10}), but we give a proof in appendix for completeness, especially since there is a slight technical difficulty when $a<0$ or $b<0$.
\begin{lemma}[Rellich-Pohozaev identity]\label{pohozaev}
Let $a_1, a_2\in {\mathbb R}$ satisfy $a_1+a_2=N-2$ and $(u,v)$ solution of (\ref{1}), there holds 
\begin{align*}
\bigg(&\frac{N+a}{p+1}-a_1\bigg)\int\limits_{B_R}|x|^av^{p+1}dx+\bigg(\frac{N+b}{q+1}-a_2\bigg)\int\limits_{B_R}|x|^bu^{q+1}dx\\
=&R^{1+b} \int\limits_{|x|=R}\frac{u^{q+1}}{q+1}d\sigma_R+R^{1+a} \int\limits_{|x|=R}\frac{v^{p+1}}{q+1}d\sigma_R\\
&+ R\int\limits_{|x|=R}\left(u'v'-\nabla u.\nabla v\right) d\sigma_R + \int\limits_{|x|=R}\big(a_1u'v+a_2uv'\big)\sigma_R
\end{align*}
where $'=\frac{1}{|x|}x.\nabla=\partial /\partial r$.
\end{lemma}

\vskip 0.2cm

We next prove an important comparison property for system (\ref{1}) under condition on the difference $a-b$. 
We follow  the ideas  of Bidaut-V\'eron in \cite{BV99} and Souplet in \cite{Sou09}.
\begin{lemma}[Comparison property]\label{comparison}
Let $pq>1$, $N\geq 3$ and $(u,v)$ be a positive solution of (\ref{1}). Assume (\ref{40}). Then 
\begin{align}\label{compPpty}
|x|^a\frac{v^{p+1}}{p+1}\leq |x|^b\frac{u^{q+1}}{q+1}.
\end{align}
\end{lemma}
\begin{proof}
Let $\sigma:=(q+1)/(p+1)\in (0,1]$, $l:=\sigma^{-1/(p+1)}$, $h:=(a-b)/(p+1)$ and $w:= v-l|x|^{-h} u^\sigma$.   
For all $x\neq 0$, we have
\begin{align*}
\Delta w= \Delta v- l\sigma |x|^{-h}u^{\sigma-1}\Delta u +l |x|^{-h}u^\sigma K,
\end{align*}
where
\begin{align*}
K=\frac{h(N-2-h)}{|x|^2}+\sigma(1-\sigma)\frac{|\nabla u|^2}{u^2}+ 2h\sigma \frac{x}{|x|^2}
\cdot\frac{\nabla u}{u}.
\end{align*}

If  $\sigma=1$ then $h=0$, thus $K=0$.

If $\sigma\in (0,1)$ then it follows from  (\ref{40}) that
\begin{align*}
K=h\Bigl(N-2-\frac{h}{1-\sigma}\Bigr)\frac{1}{|x|^2}+\sigma(1-\sigma)\Big|\frac{\nabla u}{u}
+ \frac{h}{1-\sigma}\frac{x}{|x|^2}\Big|^2 \geq 0.
\end{align*}
Hence, 
\begin{align*}
\Delta w &\geq \Delta v - l\sigma |x|^{-h}u^{\sigma-1}\Delta u\\
&=|x|^{a-h} u^{\sigma-1}\bigg((v/l)^p- \Big(|x|^{-h}u^\sigma\Big)^p\bigg).
\end{align*}
It follows that 
\begin{equation}\label{DeltaComp}
\Delta w\geq 0 \text{ \;in the set } \{x\in {\mathbb R}^N\setminus\{0\};\ w(x)\geq 0\}.
\end{equation}

If $p\geq 2$,  then for any $R>0$ and $\eps\in (0,R)$, we have
\begin{align*}
\int\limits_{B_R\setminus B_\eps} |\nabla w_{+}|^2dx
=-\int\limits_{B_R\setminus B_\eps}w_{+}\Delta w dx
+\int\limits_{|x|=R}w_{+}\partial_\nu w\, d\sigma_R
+\int\limits_{|x|=\eps}w_{+}\partial_\nu w\, d\sigma_\eps
\end{align*}
Using (\ref{DeltaComp}), the boundedness of $w_+$ near $x=0$, and passing to the limit with $\eps=\eps_i\to 0$, where $\eps_i$ is given by Lemma~\ref{lemDistrib}, we deduce that
\begin{align}\label{3}
\int\limits_{B_R} |\nabla w_{+}|^2dx\leq R^{N-1}\int\limits_{S^{N-1}}w_{+}(R)w_r(R)d\theta\leq \frac{R^{N-1}}{2}f_1'(R),
\end{align}
where $f_1(R):=\int_{S^{N-1}}(w_+)^2(R)d\theta$. 

On the other hand, let $g(R)=\int_{S^{N-1}}v^p(R)d\theta$ and note that $f_1\leq Cg^{2/p}$.
Lemma \ref{lem} guarantees that 
 $$\int_{R/2}^Rg(r)r^{N-1}dr \leq C R^{N-2-\alpha -\frac{a+bp}{pq-1}}.$$
Therefore $g(R_i)\to 0$ for some sequence $R_i\to \infty$. Consequently, $f_1(R_i)\to 0$ and there exists a sequence $\tilde{R}_i\to \infty$ such that $f_1'(\tilde{R}_i)\leq 0$. Letting $i\to \infty $ in (\ref{3}) with $R=\tilde{R}_i$, we conclude that $w_+$ is constant in ${\mathbb R}^N$. If $w=C>0$ then $v\geq C>0$ in ${\mathbb R}^N$, contradicting 
Lemma~\ref{lem}. Hence, $w_+=0$.

If $1<p< 2$,  then for any $R>0$, $\eps\in (0,R)$ and $\eta>0$, we have
\begin{align*}
(p-1)\int\limits_{B_R\setminus B_\eps}  (w_{+}+\eta)^{p-2} & |\nabla w_{+}|^2dx
=-\int\limits_{B_R\setminus B_\eps}(w_{+}+\eta)^{p-1}\Delta w dx \\
&+\int\limits_{|x|=R}(w_{+}+\eta)^{p-1}\partial_\nu w\, d\sigma_R
+\int\limits_{|x|=\eps}(w_{+}+\eta)^{p-1}\partial_\nu w\, d\sigma_\eps.
\end{align*}
Letting $\eta\to 0$ (passing to the limit in the LHS via monotone convergence) and using (\ref{DeltaComp}), it follows that
$$
(p-1)\int\limits_{B_R\setminus B_\eps}  w_{+}^{p-2}  |\nabla w_{+}|^2dx
\leq \int\limits_{|x|=R} w_{+}^{p-1}\partial_\nu w\, d\sigma_R
+\int\limits_{|x|=\eps} w_{+}^{p-1}\partial_\nu w\, d\sigma_\eps.
$$
Next passing to the limit with $\eps=\eps_i\to 0$, where $\eps_i$ is given by Lemma~\ref{lemDistrib}, we deduce that
\begin{align}\label{3z}
(p-1)\int\limits_{B_R} w^{p-2}_{+}|\nabla w_{+}|^2dx
\leq R^{N-1}\int\limits_{S^{N-1}}w^{p-1}_{+}(R)w_r(R)d\theta
\leq \frac{R^{N-1}}{p}f_2'(R),
\end{align}
where $f_2(R):=\int_{S^{N-1}}(w_+)^p(R)d\theta$. 
Using $f_2\leq g$ and arguing as above, we have $w_+=0$.
\end{proof}
\vskip 0.2cm 
\section{Proof of Theorem \ref{cond}}
 We first prove the theorem for dimension $N\geq 4$. The proof consists of 6 steps similar to those in \cite{Sou09}.
 We repeat these steps in detail for completeness 
 and because of the additional technicalities introduced by the coefficient $|x|^a, |x|^b$.
 Suppose that there exists a positive 
  solution $(u,v)$ of (\ref{1}) in ${\mathbb R}^N$.

\smallskip
\noindent{\bf Step 1:} {\it Preparations. } Let us choose $a_1, a_2$ such that
\begin{equation}\label{choicea1a2}
\frac{N+a}{p+1}>a_1, \; \frac{N+b}{q+1}>a_2
\end{equation}
and set $F(R)=\int_{B_R}|x|^bu^{q+1}$.
By the Rellich-Pohozaev identity (Lemma \ref{pohozaev}) and the comparison property (\ref{compPpty}), we have
\begin{align*}
F(R)\leq C\big(G_1(R)+G_2(R)\big),
\end{align*}
where 
\begin{align}
 &G_1(R)=R^{N+b}\int\limits_{S^{N-1}}u^{q+1}(R)d\theta,\\
&G_2(R)=R^N\int\limits_{S^{N-1}}\left(|D_xu(R)|+R^{-1}u(R) \right)\left(|D_xv(R)|+R^{-1}v(R) \right)d\theta.
\end{align}
We may assume that 
\begin{align}\label{condp}
p\geq \frac{N+2}{N-2}.
\end{align}
In fact, if $q\le p< \frac{N+2}{N-2}$, then we may apply Theorem \ref{cond2}
(which will be proved independently of Theorem \ref{cond} in Section 5).

\smallskip

\noindent{\bf Step 2:} {\it Estimation of $G_1(R)$.}
Let
\begin{align}
 \lambda=\frac{N-1}{N-3}, \; k=\frac{p+1}{p} \quad \text{ and }\quad \varepsilon >0.
\end{align}
(The number $\varepsilon$ will be ultimately chosen small; in what follows, the constant $C$ may depend on $\varepsilon$.) By the Lemma \ref{1b}, we have
\begin{align*}
&\|u\|_\lambda\leq C\left(\|D_\theta^2 u\|_{1+\varepsilon}+ \|u\|_1\right)\leq C\left(R^2\|D_x^2 u\|_{1+\varepsilon}+ \|u\|_1\right).
\end{align*}
We show that 
\begin{align}\label{1e}
 \frac{1}{k}-\frac{1}{q+1} <\frac{2}{N-1}.
\end{align}
Indeed
\begin{align*}
 \frac{1}{k}-\frac{1}{q+1}=\frac{pq-1}{(p+1)(q+1)}=\frac{2}{(p+1)\beta}=\frac{2}{\alpha+\beta+2}<\frac{2}{N-1}.
\end{align*}
On the other hand, from (\ref{condp}), there exists $\mu>0$ such that
 \begin{align*}
\frac{1}{k}- \frac{1}{\mu} =\frac{2}{N-1}.
\end{align*}
It follows from (\ref{1e}) that $\mu>q+1$. 
By 
Lemma \ref{1b}, we have
\begin{align*}
\|u\|_\mu\leq C\left(\|D_\theta^2 u\|_k+ \|u\|_1\right)\leq C\left(R^2\|D_x^2 u\|_k+ \|u\|_1\right).
\end{align*}
If $\lambda < q+1$ then
\begin{align}\label{1c}
 \|u\|_{q+1}\leq \|u\|_\lambda^\nu\|u\|_\mu^{1-\nu}\leq C\left(R^2\|D_x^2 u\|_{1+\varepsilon}+ \|u\|_1\right)^\nu \left(R^2\|D_x^2 u\|_k+ \|u\|_1\right)^{1-\nu}.
\end{align}
If $\lambda\geq q+1$ then (\ref{1c}) is still valid with $\nu=1$. In both cases, we see that $\nu$ is given by
\begin{align}
 \nu=1-(p+1)A, \text{ with } A=\left(\frac{N-3}{N-1}-\frac{1}{q+1}\right)_+.
\end{align}
Therefore,
\begin{align}\label{estimG1step2}
\left[R^{-N-b}G_1(R)\right]^{1/q+1}\leq CR^2\left(\|D_x^2 u\|_{1+\varepsilon}+ R^{-2}\|u\|_1\right)^\nu \left(\|D_x^2 u\|_k+ R^{-2}\|u\|_1\right)^{1-\nu}.
\end{align}
{\bf Step 3:} {\it Estimation of $G_2(R)$.}
Let 
\begin{align*}
 m=\frac{q+1}{q}, \quad \rho=\frac{N-1}{N-2}.
\end{align*}
By Lemma \ref{1b}, we have
\begin{align}
\|D_xu\|_\rho\leq C\left(\|D_\theta D_x u\|_{1+\varepsilon}+\|D_xu\|_1\right)
 \leq C\left(R\|D^2_x u\|_{1+\varepsilon}+\|D_xu\|_1\right).
\end{align}

{\it Case 1.} $q>1/(N-2)$. Let $\gamma_1, \gamma_2$ be defined by
\begin{align*}
 \frac{1}{\gamma_1}=\frac{p}{p+1}-\frac{1}{N-1}, \quad \frac{1}{\gamma_2}=\frac{q}{q+1}-\frac{1}{N-1}.
\end{align*}
Then we have

\begin{align*}
k<\gamma_1<\infty, \quad m<\gamma_2<\infty.
\end{align*}
Assume that we can find $z\in (1,\infty)$ such that
\begin{align}\label{16}
 \frac{1}{k}-\frac{1}{N-1}\leq \frac{1}{z}\leq 1-\frac{1}{N-1}
\end{align}
and 
\begin{align}\label{17}
 \frac{1}{m}-\frac{1}{N-1}\leq 1-\frac{1}{z}\leq 1-\frac{1}{N-1}.
\end{align}

By the same estimate as in \cite{Sou09}, we have

\begin{equation}\label{1d}
 \begin{split}
 G_2(R)\le C &R^{N+2}\left(\|D_x^2u\|_{1+\varepsilon}+ R^{-1}\|D_xu\|_1+ R^{-2}\|u\|_1\right)^{\tau_1}\\
&\times \left(\|D_x^2u\|_k+ R^{-1}\|D_xu\|_1+ R^{-2}\|u\|_1\right)^{1-\tau_1}\\
&\times \left(\|D_x^2v\|_{1+\varepsilon}+ R^{-1}\|D_xv\|_1+ R^{-2}\|v\|_1\right)^{\tau_2}\\
&\times \left(\|D_x^2v\|_m+ R^{-1}\|D_xv\|_1+ R^{-2}\|v\|_1\right)^{1-\tau_2}.
\end{split}
\end{equation}
where
\begin{align*}
 \tau_1=1-(p+1)A_1, \quad A_1=\frac{N-2}{N-1}-\frac{1}{z},\\
 \tau_2=1-(q+1)A_2, \quad A_2=\frac{1}{z}-\frac{1}{N-1}.
\end{align*}

{\it Case 2.}  $q\leq 1/(N-2)$. Then (\ref{1d}) remains true with $\tau_1=1, \tau_2=0$.

\medskip
\noindent{\bf Step 4:} {\it Control the averages.}
For any $R>1$ we claim that
\begin{align}\label{22}
\begin{cases}
\int_{R/2}^R\|u(r)\|_1r^{N-1}dr\leq CR^{N-\alpha- \frac{a+bp}{pq-1}},\\ \int_{R/2}^R\|v(r)\|_1r^{N-1}dr\leq CR^{N-\beta-\frac{b+aq}{pq-1}},
\end{cases}
\end{align}

\begin{align}\label{23}
\begin{cases}
\int_{R/2}^R\|D_xu(r)\|_1r^{N-1}dr\leq CR^{N-1-\alpha- \frac{a+bp}{pq-1}},\\
\int_{R/2}^R\|D_xv(r)\|_1r^{N-1}dr\leq CR^{N-1-\beta-\frac{b+aq}{pq-1}},
\end{cases}
\end{align}
\begin{align}\label{24a}
\int_{R/2}^R\|D_x^2u(r)\|^k_k r^{N-1}dr\leq CR^{\frac{a}{p}}F(2R),
\end{align}
\begin{align}\label{24b}
\int_{R/2}^R\|D^2_xv(r)\|^m_m r^{N-1}dr\leq CR^{\frac{b}{q}}F(2R),
\end{align}

\begin{align}\label{25}
\begin{cases}
\int_{R/2}^R\|D_x^2u(r)\|_{1+\varepsilon}^{1+\varepsilon} r^{N-1}dr\leq CR^{N-2-\alpha- \frac{a+bp}{pq-1}+a\varepsilon},\\
\int_{R/2}^R\|D^2_xv(r)\|_{1+\varepsilon}^{1+\varepsilon} r^{N-1}dr\leq CR^{N-2-\beta-\frac{b+aq}{pq-1}+b\varepsilon},
\end{cases}
\end{align}
Estimates (\ref{22}) and (\ref{23}) follow from Lemma \ref{lem}. Let us next prove (\ref{24a}), (\ref{24b}) and (\ref{25}). Indeed,
\begin{align*}
\int_{R/2}^R\|D_x^2u(r)\|^k_k r^{N-1}dr&=\int\limits_{B_R\setminus B_{R/2}}|D_x^2u|^kdx\\
&\leq C\left(\int\limits_{B_{2R}\setminus B_{R}}|\Delta u|^kdx+ R^{-2k}\int\limits_{B_{2R}\setminus B_{R}}u^kdx\right)\\
&= C\left(\int\limits_{B_{2R}\setminus B_{R}}|x|^{ka}v^{p+1}dx+ R^{-2k}\int\limits_{B_{2R}\setminus B_{R}}u^kdx\right)\\
&\leq  C\left(R^{a/p}F(2R)+ R^{-2k}\int\limits_{B_{2R}\setminus B_{R}}u^kdx\right).
\end{align*}
By H\"older's inequality, for $R>1$, we have
\begin{align*}
A_1=R^{-2k}\int\limits_{B_{2R}\setminus B_{R}}u^kdx&\leq CR^{-2k}R^{N(pq-1)/p(q+1)}\left(\int\limits_{B_{2R}\setminus B_{R}}u^{q+1}dx\right)^{(p+1)/p(q+1)}\\
&\leq R^{\eta_1/p}F(2R),
\end{align*}
with $\eta_1=-2(p+1)+N(pq-1)/(q+1)- b(p+1)/(q+1)$, where we used $(p+1)/p(q+1)<1$, along with 
\begin{align*}
F(R)\geq F(1)>0, \quad R>1.
\end{align*}
We show that $\eta_1<a$. Indeed
\begin{align*}
a-\eta_1 &=2(p+1)-N\frac{pq-1}{q+1}+ b\frac{p+1}{q+1}+a\\
&=2(p+1)-N\frac{pq-1}{q+1}+ b\frac{p(q+1)-(pq-1)}{q+1}+a\\
&=\frac{2}{\beta}\left((p+1)\beta-N+ \frac{b}{2}p\beta-b+\frac{a}{2}\beta\right)\\
&=\frac{2}{\beta}\left(2+\alpha+\beta-N+ \frac{b}{2}\alpha+\frac{a}{2}\beta\right).
\end{align*}
Hence (\ref{24a}) holds. The similar argument and Lemma \ref{comparison} imply (\ref{24b}).

On the other hand, by using Lemma \ref{lemell}, \ref{lem4a}, equation (\ref{1}) and the boundedness of $u$, we obtain
\begin{align*}
 \int\limits_{R/2}^R\|D_x^2u(r)\|_{1+\eps}^{1+\eps} r^{N-1}\,dr&=\int\limits_{B_R\setminus B_{R/2}}|D_x^2u|^{1+\eps}\,dx\\
&\leq 
C\int\limits_{B_{2R}\setminus B_{R/4}}|\Delta u|^{1+\eps}\,dx
+CR^{-2(1+\eps)} \int\limits_{B_{2R}\setminus B_{R/4}}u^{1+\eps}\,dx\\
&\leq 
C\int\limits_{B_{2R}\setminus B_{R/4}}|x|^{a\eps}u^{p\eps} |x|^av^p\,dx
+CR^{-2(1+\eps)}\int\limits_{B_{2R}\setminus B_{R/4}}u^{1+\eps}\,dx\\
&\leq 
CR^{a\eps}\int\limits_{B_{2R}\setminus B_{R/4}}|x|^av^pdx
+CR^{-2(1+\eps)}\int\limits_{B_{2R}\setminus B_{R/4}}u\,dx\\
&\leq CR^{N-2-\alpha-\frac{a+bp}{pq-1}+ a\eps}+CR^{N-2-\alpha-\frac{a+bp}{pq-1}-2\eps} \\
&\leq CR^{N-2-\alpha-\frac{a+bp}{pq-1}+ a\eps}.
\end{align*}
By the similar calculation for $v$,  (\ref{25}) holds.

\smallskip
\noindent{\bf Step 5:} {\it measure and feedback argument.} For a given $K>0$, let us define the sets
\begin{align*}
&\Gamma_1(R)=\{r\in (R,2R); \|D^2_xu(r)\|_k^k>KR^{-N+\frac{a}{p}}F(4R)\},\\
&\Gamma_2(R)=\{r\in (R,2R); \|D^2_xv(r)\|_m^m>KR^{-N+\frac{b}{q}}F(4R)\},\\
&\Gamma_3(R)=\{r\in (R,2R); \|D^2_xu(r)\|_{1+\varepsilon}^{1+\varepsilon}>KR^{-2-\alpha-\frac{a+bp}{pq-1}+a\varepsilon}\},\\
&\Gamma_4(R)=\{r\in (R,2R); \|D^2_xv(r)\|_{1+\varepsilon}^{1+\varepsilon}>KR^{-2-\beta-\frac{b+aq}{pq-1}+b\varepsilon}\},\\
&\Gamma_5(R)=\{r\in (R,2R); \|u(r)\|_1>KR^{-\alpha-\frac{a+bp}{pq-1}}\},\\
&\Gamma_6(R)=\{r\in (R,2R); \|v(r)\|_1>KR^{-\beta-\frac{b+aq}{pq-1}}\},\\
&\Gamma_7(R)=\{r\in (R,2R); \|D_xu(r)\|_1>KR^{-1-\alpha-\frac{a+bp}{pq-1}}\},\\
&\Gamma_8(R)=\{r\in (R,2R); \|D_xv(r)\|_1>KR^{-1-\beta-\frac{b+aq}{pq-1}}\}.
\end{align*}
By 
estimate (\ref{24a}) and (\ref{22}), for $R>1$ we have
\begin{align*}
CR^{a/p}F(4R)&\geq \int_{R}^{2R}\|D^2_xu(r)\|^k_kr^{N-1}dr\\
&\geq |\Gamma_1(R)|R^{N-1}KR^{-N+\frac{a}{p}}F(4R)=|\Gamma_1(R)|KR^{-1+\frac{a}{p}}F(4R)
\end{align*}
and 
\begin{align*}
C&\geq R^{-N+\alpha +\frac{a+bp}{pq-1}}\int_{R}^{2R}\|u(r)\|_1r^{N-1}dr\\
&\geq R^{-N+\alpha +\frac{a+bp}{pq-1}} |\Gamma_5(R)|R^{N-1}KR^{-\frac{a+bp}{pq-1}}=|\Gamma_5(R)|KR^{-1}.
\end{align*}
Consequently, $|\Gamma_1|\leq R/10$ and $|\Gamma_5|\leq R/10$ for $K>10C$. Similarly, $|\Gamma_i|\leq R/10,\ i=1,...,8$. Therefore, for each $R\geq 1$, we can find 
\begin{align}
 \tilde{R}\in (R, 2R)\setminus\bigcup_{i=1}^8 \Gamma_i(R)\ne \emptyset.
\end{align}

Let us check that  
\begin{align}
 2+\alpha+\frac{a+bp}{pq-1}> \frac{N}{k}-\frac{a}{pk}\label{10},\\
 2+\beta+\frac{b+aq}{pq-1}> \frac{N}{m}-\frac{b}{qm}.\label{11}
\end{align}
Indeed, by computation
\begin{align*}
M&= (2+\alpha)k+\frac{a+bp}{pq-1}k-N+\frac{a}{p}\\
&=p\beta k+\frac{(a+bp)(p+1)}{p(pq-1)}-N+\frac{a}{p}\\
&=\beta(p+1)+ \frac{(a+bp)}{2p}\alpha -N+\frac{a}{p}\\
&=p\beta+\beta+\frac{a}{2p}\alpha+ \frac{b}{2}\alpha-N+\frac{a}{p}\\
&=\alpha+\beta+2-N+\frac{b}{2}\alpha+\frac{a}{2p}(\alpha+2)\\
&=\alpha+\beta+2-N+\frac{b}{2}\alpha+\frac{a}{2}\beta>0.
\end{align*}
Thus, (\ref{10}) holds. Similarly for (\ref{11}).
Therefore, for $\varepsilon >0$ small enough, we have
\begin{align}
 \frac{1}{1+\varepsilon}(2+\alpha+\frac{a+bp}{pq-1}-a\varepsilon)> \frac{N}{k}-\frac{a}{pk}\label{10a},\\
 \frac{1}{1+\varepsilon}(2+\beta+\frac{b+aq}{pq-1}-b\varepsilon)> \frac{N}{m}-\frac{b}{qm}.\label{11a}
\end{align}

By (\ref{estimG1step2}) and the definition of the sets $\Gamma_i$, we may now control $G_1(\tilde R)$ as follows 
\begin{align*}
\left[R^{-N-b}G_1(\tilde R)\right]^{1/q+1}\leq CR^2\left(R^{-2-\alpha-\frac{a+bp}{pq-1}}+ R^{(-2-\alpha-\frac{a+bp}{pq-1}+a\varepsilon)/(1+\varepsilon}\right)^\nu\\
\times\left(R^{-\frac{N}{k}+\frac{a}{pk}}F^{1/k}(4R)+R^{-2-\alpha-\frac{a+bp}{pq-1}}\right)^{1-\nu}
\end{align*}
Using (\ref{10}) and (\ref{10a}), we obtain
\begin{align}\label{30}
G_1(\tilde{R})\leq C \left(R^{-a_1(0)}+R^{-a_1(\varepsilon)}\right)F^{b_1}(4R)
\end{align}
where
\begin{align*}
&a_1(\varepsilon)=(q+1)\left[\left(\alpha+2+\frac{a+bp}{pq-1}-a\varepsilon\right)\frac{\nu}{1+\varepsilon}+ \left(\frac{N}{k}-\frac{a}{pk}\right)(1-\nu)-2-\frac{N+b}{q+1}\right]\\
&b_1=\frac{(1-\nu)}{k}(q+1).
\end{align*}
On the other hand, it follows from (\ref{1d}), (\ref{10})-(\ref{11a}) that
\begin{align}
 G_2(\tilde{R})\leq C&\left(R^{-\alpha-2 -\frac{a+bp}{pq-1}}+R^{(-2-\alpha-\frac{a+bp}{pq-1}+a\varepsilon)/(1+\varepsilon}\right)^{\tau_1}\notag\\
&\times\left(R^{-\beta-2 -\frac{b+aq}{pq-1}}+R^{(-2-\beta-\frac{b+aq}{pq-1}+b\varepsilon)/(1+\varepsilon}\right)^{\tau_2}\notag\\
&\times\left(R^{-\frac{N}{k}+\frac{a}{pk}}F^{1/k}(4R)+R^{-\alpha-2 -\frac{a+bq}{pq-1}}\right)^{1-\tau_1}\notag\\
&\times\left(R^{-\frac{N}{m}+\frac{b}{qm}}F^{1/m}(4R)+R^{-\beta-2 -\frac{b+aq}{pq-1}}\right)^{1-\tau_2}\notag\\
&\leq C \left(R^{-a_2(0)}+R^{-a_2(\varepsilon)}+R^{-a_3(\varepsilon)}+R^{-a_4(\varepsilon)}\right)F^{b_2}(4R).\label{31}
\end{align}
where
\begin{align*}
a_2(\varepsilon)=&-N-2+\frac{\tau_1}{1+\varepsilon}\left(\alpha+2 +\frac{a+bp}{pq-1}-a\varepsilon\right)+ \tau_2\left(\beta+2 +\frac{b+aq}{pq-1}\right)\notag\\
&+\left(N-\frac{a}{p}\right)\frac{(1-\tau_1)}{k} +
\left(N-\frac{b}{q}\right)\frac{(1-\tau_2)}{m},
\end{align*}
\begin{align*}
a_3(\varepsilon)=&-N-2+\tau_1\left(\alpha+2 +\frac{a+bp}{pq-1}\right)+ \frac{\tau_2}{1+\varepsilon}\left(\beta+2 +\frac{b+aq}{pq-1}-b\varepsilon\right)\notag\\
&+\left(N-\frac{a}{p}\right)\frac{(1-\tau_1)}{k} +
\left(N-\frac{b}{q}\right)\frac{(1-\tau_2)}{m},
\end{align*}
\begin{align*}
a_4(\varepsilon)=&-N-2+\frac{\tau_1}{1+\varepsilon}\left(\alpha+2 +\frac{a+bp}{pq-1}-a\varepsilon\right)+ \frac{\tau_2}{1+\varepsilon}\left(\beta+2 +\frac{b+aq}{pq-1}-b\varepsilon\right)\notag\\
&+\left(N-\frac{a}{p}\right)\frac{(1-\tau_1)}{k} +
\left(N-\frac{b}{q}\right)\frac{(1-\tau_2)}{m},
\end{align*}
\begin{align*}
 b_2=&\frac{1-\tau_1}{k}+\frac{1-\tau_2}{m}.
\end{align*}

Let $\tilde{a}=\min\left(a_i(0), a_j(\varepsilon); \quad i=1,2;\quad j=1,2,3,4\right)$ and $\tilde{b}=\max(b_1, b_2)$. Combining (\ref{30})  and (\ref{31}), we obtain
\begin{align}\label{32}
 F(R)\leq C R^{-\tilde{a}}F^{\tilde{b}}(4R), R\geq 1.
\end{align}
We claim that there exist a constant $M>0$ and a sequence $R_i\to \infty$ such that 
\begin{align*}
 F(4R_i)\leq MF(R_i).
\end{align*}
Assume that the claim is false. Then, for any $M>0$, there exists $R_0>0$ such that $F(4R)\geq MF(R)$ for all $R\geq R_0$. But since $u$ is bounded, we have $F(R)\leq CR^{N+b}$. Thus 
$$M^iF(R_0)\leq F(4^iR_0)\leq C(4^iR_0)^{N+b}= CR_0^{N+b}4^{i(N+b)}, \forall i\geq 0.$$
This is a contradiction for $i$ large if we choose $M>4^{N+b}$.

Now we assume we have proved that $\tilde{a}>0$ and $\tilde{b}<1$, then from (\ref{32}) we have
$$F(4R_i)\leq C R_i^{-\tilde{a}/(1-\tilde{b})}. $$
Letting $i\to \infty$, we obtain $\int_{{\mathbb R}^N}|x^b|u^{q+1}=0$, hence $u\equiv 0\equiv v$: contradiction.

\medskip
\noindent{\bf Step 6:} {\it Fulfillment of the conditions $\tilde{a}>0$ and $\tilde{b}<1$}
\smallskip

{\bf Verification of $b_1<1$.} If $q\leq 2/(N-3)$ then $b_1=0$. If $q>2/(N-3)$ then
\begin{align*}
 1-b_1&=1-p(q+1)A=1-p\left((q+1)\frac{N-3}{N-1}-1\right)=\frac{(N-1)(p+1)-p(q+1)(N-3)}{N-1}\\
&=\frac{2(p+1)-(N-3)(pq-1)}{N-1}=\frac{pq-1}{N-1}(\alpha+3-N).
\end{align*}
Thus, $0\leq b_1<1$.
\smallskip

{\bf Verification of $a_1(0)>0$.}
\begin{align*}
a_1(0)&=(q+1)\left[\alpha+\frac{a+bp}{pq-1}-\frac{N+b}{q+1}- (1-\nu)\frac{1}{k}\left((2+\alpha)k+\frac{a+bp}{pq-1}k-N-\frac{a}{p}\right)\right]\\
&=(q+1)\alpha+ \frac{(q+1)(a+bp)}{pq-1}-N-b-b_1M\\
&= \alpha+\beta+2+\frac{(q+1)(a+bp)}{pq-1}-N-b-b_1M\\
&=M-b_1M=(1-b_1)M.
\end{align*}
Hence $a_1(0)>0$.
\smallskip

{\bf Verification of $a_2(0)>0$ and $b_2<1$.}

{\bf  Case $q>1/(N-2)$.} Here we must ensure the existence of $z\in (1, \infty)$ satisfying (\ref{16}) and (\ref{17}), that is
\begin{align}\label{18}
\max\bigg(\frac{1}{k}-\frac{1}{N-1},\frac{1}{N-1}\bigg)\leq\frac{1}{z}\leq
\min\bigg(1-\frac{1}{N-1}, \frac{1}{q+1}+\frac{1}{N-1}\bigg). 
\end{align}

We have 
\begin{align*}
 b_2=pA_1+qA_2=p\left(\frac{N-2}{N-1}-\frac{1}{z}\right)+q\left(\frac{1}{z}-\frac{1}{N-1}\right)=\frac{p(N-2)-q}{N-1}-\frac{p-q}{z}.
\end{align*}
Hence, there exists $z\in (1,\infty)$ satisfying (\ref{18}) and such that $b_2<1$, if the following 
 hold

\begin{align}
&\frac{1}{k}-\frac{1}{N-1}\leq \frac{1}{q+1}+\frac{1}{N-1},\label{19}\\
&\frac{p(N-2)-q}{N-1}-1<\frac{(N-2)(p-q)}{N-1},\label{20}\\
&\frac{p(N-2)-q}{N-1}-1<(p-q)\left(\frac{1}{q+1}+\frac{1}{N-1}\right).\label{21}
\end{align}
Inequality (\ref{19}) is true by (\ref{1e}). Inequality (\ref{20}) is equivalent to $q<(N-1)/(N-3)$, which is true due to $q\leq p(q+1)/(p+1)=1+(2/\alpha)<(N-1)/(N-3)$. Inequality (\ref{21}) is also true due to $\alpha>N-3$. 

We have
\begin{align*}
a_2(0)=&-N-2+\tau_1\left(\alpha+2 +\frac{a+bp}{pq-1}\right)+ \tau_2\left(\beta+2 +\frac{b+aq}{pq-1}\right)\notag\\
&+\left((2+\alpha)k+\frac{a+bp}{pq-1}k-M \right)\frac{(1-\tau_1)}{k} \\
&+\left((2+\beta)m+\frac{b+aq}{pq-1}m-M\right)\frac{(1-\tau_2)}{m}\\
=&-N-2+2+\alpha+\frac{a+bp}{pq-1}+2+\beta+\frac{b+aq}{pq-1}-M\left(\frac{1-\tau_1}{k}+\frac{1-\tau_2}{m}\right)\\
=& M- Mb_2 = M(1-b_2)>0.
\end{align*}

{\bf  Case $q\leq 1/(N-2)$.} Since $\tau_1=1, \tau_2=0$, we deduce
\begin{align*}
a_2(0)=& -N-2+\alpha+ 2+\frac{a+bp}{pq-1}+\frac{N}{m}-\frac{b}{qm}\\
=&\frac{1}{q+1}\left(-N+(q+1)\alpha+\frac{(q+1)(a+bp)}{pq-1}-\frac{b(pq-1)}{pq-1}\right)\\
=&\frac{1}{q+1}\left(\alpha+\beta+\frac{b}{2}\alpha+\frac{a}{2}\beta-N+2\right)>0
\end{align*}
and also $b_2=1/m<1$. 

Note that $a_2(0)=a_3(0)=a_4(0)$. Thus $a_i(\varepsilon)>0, \ i=1,...,4$ for $\varepsilon$ small enough. Theorem is proved for $N\geq 4$.

\medskip

For $N=3$,  conditions (\ref{40}) and (\ref{50}) are not necessary and the proof becomes much less complicated due to the Sobolev imbedding $W^{2, 1+\varepsilon}\subset L^\infty$ on $S^2$. For sake of clarity, although here $N=3$, we shall keep the letter $N$ in the proof.

\medskip

\noindent{\bf Step 1:} {\it Preparations. } 
Let us choose $a_1, a_2$ satisfying (\ref{choicea1a2}) and set
$$F(R)=\int_{B_R}|x|^au^{q+1}dx+ \int_{B_R}|x|^bv^{p+1}dx.$$
By the Rellich-Pohozaev identity (Lemma \ref{pohozaev}), we have
\begin{align*}
F(R)\leq C\big(G_{11}(R)+G_{12}(R)+G_2(R)\big),
\end{align*}
where 
\begin{align*}
 &G_{11}(R)=R^{N+b}\int\limits_{S^{N-1}}u^{q+1}(R)d\theta,\\
&G_{12}(R)=R^{N+a}\int\limits_{S^{N-1}}v^{p+1}(R)d\theta,\\
&G_2(R)=R^N\int\limits_{S^{N-1}}\left(|D_xu(R)|+R^{-1}u(R) \right)\left(|D_xv(R)|+R^{-1}v(R) \right)d\theta.
\end{align*}

\medskip
\noindent{\bf Step 2:} {\it Estimations of $G_{11}(R), G_{12}(R)$ and $G_2(R)$.}

By Lemma \ref{1b}, since $N=3$, we have 
$$\|u\|_{q+1}\leq \|u\|_\infty\leq C\left(\|D_\theta^2 u\|_{1+\eps}+ \|u\|_1\right)\leq 
C\left(R^2\|D_x^2 u\|_{1+\eps}+ \|u\|_1\right)$$
and
$$\|D_xu\|_2\leq C\left(\|D_\theta D_x u\|_{1+\eps}+\|D_xu\|_1\right)
 \leq C\left(R\|D^2_x u\|_{1+\eps}+\|D_xu\|_1\right).$$
Similarly,
$$\|v\|_{p+1}\leq \|v\|_\infty\leq C\left(\|D_\theta^2 v\|_{1+\eps}+ \|v\|_1\right)\leq 
C\left(R^2\|D_x^2 v\|_{1+\eps}+ \|v\|_1\right)$$
and
$$\|D_xv\|_2\leq C\left(\|D_\theta D_x v\|_{1+\eps}+\|D_xv\|_1\right)
 \leq C\left(R\|D^2_x v\|_{1+\eps}+\|D_xv\|_1\right).$$
Therefore,
\begin{equation}\label{estimG11}
G_{11}(R)\leq CR^{N+b+2(q+1)}\left(\|D_x^2 u\|_{1+\eps}+ R^{-2}\|u\|_1\right)^{q+1},
\end{equation}
\begin{equation}\label{estimG12}
G_{12}(R)\leq CR^{N+a+2(p+1)}\left(\|D_x^2 v\|_{1+\eps}+ R^{-2}\|v\|_1\right)^{p+1}
\end{equation}
and
\begin{equation}\label{estimG2}
G_2(R)\leq CR^{N+2}(\|D^2_xu\|_{1+\eps}+R^{-1}\|D_xu\|_1+ R^{-2}\|u\|_1). (\|D^2_xv\|_{1+\eps}+R^{-1}\|D_xv\|_1+ R^{-2}\|v\|_1)
\end{equation}

\medskip

\medskip
\noindent{\bf Step 3:} {\it Conclusion.} 
We can find 
\begin{equation}\label{existRtilda}
\tilde{R}\in (R, 2R)\setminus\bigcup_{i=3}^{8} \Gamma_i(R)\ne \emptyset,
\end{equation}
where the sets $\Gamma_i$ are defined in Step 4 of the proof of the case $N\ge 4$.
If follows from (\ref{estimG11})-(\ref{estimG2}) in Step 2 and (\ref{existRtilda}) in Step 3 that 
\begin{align*}
G_{11}(\tilde{R})
&\leq CR^{N+b+2(q+1)}\left(R^{(-2-\alpha-\frac{a+bp}{pq-1}+a\varepsilon)/(1+\eps)}+R^{-2-\alpha-\frac{a+bp}{pq-1}}\right)^{q+1}\\
&\leq C\left( R^{-c_1(\eps)}+R^{-c_1(0)}\right),
\end{align*}
where
\begin{align*}
c_1(\eps)=(q+1)\left[\left(2+\alpha+\frac{a+bp}{pq-1}-a\varepsilon\right)\frac{1}{1+\eps}-2-\frac{N+b}{q+1}\right].
\end{align*}
Similarly
\begin{align*}
G_{12}(\tilde{R})
&\leq CR^{N+a+2(p+1)}\left(R^{(-2-\beta-\frac{b+aq}{pq-1}+b\varepsilon)/(1+\eps)}+R^{-2-\beta-\frac{b+aq}{pq-1}}\right)^{p+1}\\
&\leq C\left( R^{-c_2(\eps)}+R^{-c_2(0)}\right),
\end{align*}
where
\begin{align*}
c_2(\eps)=(p+1)\left[\left(2+\beta+\frac{b+aq}{pq-1}-b\varepsilon\right)\frac{1}{1+\eps}-2-\frac{N+a}{p+1}\right],
\end{align*}
and
\begin{align*} 
 G_2(\tilde{R})&\leq CR^{N+2}\left(R^{(-2-\alpha-\frac{a+bp}{pq-1}+a\varepsilon)/(1+\eps)}+R^{-2-\alpha-\frac{a+bp}{pq-1}}\right)\\
&\qquad\times\left(R^{(-2-\beta-\frac{b+aq}{pq-1}+b\varepsilon)/(1+\eps)}+R^{-2-\beta-\frac{b+aq}{pq-1}}\right) \\
&\leq C\left( R^{-c_3(\eps)}+R^{-c_3(\eps)}+R^{-c_4(\eps)}+R^{-c_3(0)}\right), 
\end{align*}
where
\begin{align*}
&c_3(\eps)=-N-2+\frac{1}{1+\eps}\left(2+ \alpha+\frac{a+bp}{pq-1}-a\varepsilon\right)+\frac{1}{1+\eps}\left(2+ \beta+\frac{b+aq}{pq-1}-b\varepsilon\right),\\
&c_4(\eps)=-N-2+\frac{1}{1+\eps}\left(2+ \alpha+\frac{a+bp}{pq-1}-a\varepsilon\right)+2+ \beta+\frac{b+aq}{pq-1},\\
&c_5(\eps)=-N-2+\frac{1}{1+\eps}\left(2+ \beta+\frac{b+aq}{pq-1}-b\varepsilon\right)+2+ \alpha+\frac{a+bp}{pq-1}.
\end{align*}

Letting $\tilde{c}=\min\left(c_i(\eps), c_j(0); i=1,...,5,j=1,...,3\right)$, we obtain
$$F(R)\leq F(\tilde R)\leq C R^{-\tilde{c}}, \ \ R\geq 1.$$
By straightforward computation, we see that 
$$c_i(0)>0, \ i=1,...,5.$$ 
Therefore, for $\eps>0$ small enough,
we have $\tilde{c}>0$, so that
 $\int_{{\mathbb R}^N}\left(|x|^au^{p+1}+|x|^bv^{p+1}\right)dx=0$, hence $u\equiv v \equiv 0$: a contradiction.
The proof is complete. 
\qed

\section{Applications: Singularity and decay estimates and a priori bound}
\subsection{Singularity and decay estimates}
We now prove 
Theorem \ref{th3}. We need the following lemma.
\begin{lemma}\label{lem3}
 Assume $pq>1$, $p\geq q$, (\ref{hyper}) and  (\ref{50}).  
 Assume in addition that $c, d \in C^\gamma(\overline{B}_1)$ for some $\gamma \in (0,1]$ and 
\begin{align}\label{boundc}
\|c\|_{C^\gamma(\overline{B}_1)}\leq C_1,\; \|d\|_{C^\gamma(\overline{B}_1)}\leq C_1\; \text{ and } \; c(x)\geq C_2,\; d(x)\geq C_2,\; x\in \overline{B}_1,
\end{align}
for some constants $C_1, C_2>0$.
There exists a constant $C$, depending only on $\gamma, C_1, C_2, p, q, N$, such that, for
any nonnegative classical solution $(u,v)$ of
 \begin{align}\label{eqcup}
\begin{cases}
-\Delta u=c(x)v^p, \quad x\in B_1\\
-\Delta v=d(x)u^q, \quad x\in B_1
\end{cases}
\end{align}
$(u,v)$ satisfies 
\begin{align}
 |u(x)|^{\frac{1}{\alpha }}+|v(x)|^{\frac{1}{\beta }}
 +|\nabla u(x)|^{\frac{1}{\alpha+1}}+|\nabla v(x)|^{\frac{1}{\beta+1}}
\leq C\bigl(1+{\rm dist}^{-1}(x,\partial B_1)\bigr),\quad x\in B_1.
\end{align}

\end{lemma}

\begin{proof} 
Arguing by contradiction, we suppose that there exist 
 sequences $c_k, d_k, u_k, v_k$
verifying (\ref{boundc}), (\ref{eqcup}) and points $y_k$, such that the functions 
$$M_k=|u|^{\frac{1}{\alpha }}+|v|^{\frac{1}{\beta }}
+|\nabla u|^{\frac{1}{\alpha+1}}+|\nabla v|^{\frac{1}{\beta+1}}$$
satisfy
$$M_k(y_k)>2k\bigl(1+{\rm dist}^{-1}(y_k,\partial B_1)\bigr)\geq 2k\,{\rm dist}^{-1}(y_k,\partial B_1).$$
By the Doubling Lemma in \cite[Lemma 5.1]{PQS07}, there exists $x_k$ such that
$$M_k(x_k)\geq M_k(y_k),\quad M_k(x_k)>2k\,{\rm dist}^{-1}(x_k,\partial B_1),$$
and 
\begin{equation}\label{ineqMk}
M_k(z)\leq 2M_k(x_k), \quad\hbox{ for all $z$ such that } |z-x_k|\leq kM_k^{-1}(x_k).
\end{equation}
We have 
\begin{equation}\label{convlambda}
\lambda_k:=M_k^{-1}(x_k)\to 0,\quad k\to\infty,
\end{equation}
due to $M_k(x_k)\geq M_k(y_k)>2k$. 

Next we let 
$$\tilde{u}_k=\lambda_k^{\alpha}u_k(x_k+\lambda_k y),\;\tilde{v}_k=\lambda_k^{\beta}v_k(x_k+\lambda_k y),\quad \tilde c_k(y)=c_k(x_k+\lambda_k y), \; \tilde d_k(y)=d_k(x_k+\lambda_k y).$$
We note that $|\tilde{u}_k(0)|^{\frac{1}{\alpha}}+|\tilde{v}_k(0)|^{\frac{1}{\beta}}
+|\nabla \tilde{u}_k(0)|^{\frac{1}{\alpha+1}}+|\nabla \tilde{v}_k(0)|^{\frac{1}{\beta+1}}=1$, 
\begin{equation}\label{boundvk}
\left[|\tilde{u}_k|^{\frac{1}{\alpha}}+|\tilde{v}_k|^{\frac{1}{\beta}}\right](y)\leq 2, \quad |y|\leq k,
\end{equation}
due to (\ref{ineqMk}), and we see that $(\tilde{u}_k, \tilde{v}_k)$ satisfies
\begin{equation}\label{eqnvk}
\begin{cases}
 -\Delta \tilde{u}_k=\tilde c_k(y)\tilde{v}_k^p, \quad |y|\leq k,\\
-\Delta \tilde{v}_k=\tilde d_k(y)\tilde{u}_k^q, \quad |y|\leq k.
\end{cases}
\end{equation}

On the other hand, due to (\ref{boundc}), we have $C_2\leq \tilde c_k, \tilde d_k\leq C_1$ and, for each $R>0$ and $k\geq k_0(R)$ large enough,
\begin{equation}\label{boundAscolic}
\begin{cases}
 |\tilde c_k(y)-\tilde c_k(z)|\leq C_1|\lambda_{k}(y-z)|^\alpha\leq C_1|y-z|^\alpha,\quad |y|,|z|\leq R,\\
 |\tilde d_k(y)-\tilde d_k(z)|\leq C_1|\lambda_{k}(y-z)|^\alpha\leq C_1|y-z|^\alpha,\quad |y|,|z|\leq R.
\end{cases}
\end{equation}
Therefore, by Ascoli's theorem, there exists $\tilde c, \tilde d$ in $C({\mathbb R}^N)$  such that, 
after extracting a subsequence,
$(\tilde c_k, \tilde d_k)\to (\tilde c, \tilde d)$ in $C_{loc}({\mathbb R}^N)$. 
Moreover, (\ref{boundAscolic}) and (\ref{convlambda}) imply that $|\tilde c_k(y)-\tilde c_k(z)|\to 0$ as $k\to\infty$, so that the function $\tilde c$ is actually a constant $C\geq C_2$. Similarly, $\tilde d$ is actually a constant $D\geq C_2$.

Now, for each $R>0$ and $1<q<\infty$, by (\ref{eqnvk}), (\ref{boundvk}) and interior elliptic $L^q$ estimates, 
the sequence $(\tilde u_k, \tilde v_k)$ is uniformly bounded in $W^{2+\gamma,q}(B_R)$.  
Using standard imbeddings, after extracting a subsequence,  
we may assume that $(\tilde u_k, \tilde v_k)\to (\tilde u, \tilde v)$ in $C^2_{loc}({\mathbb R}^N)$.
It follows that $(\tilde u, \tilde v)$  is a nonnegative classical solution of
\begin{align*}
 \begin{cases}
  -\Delta \tilde u= C\tilde v^p, \quad y\in  {\mathbb R}^N,\\
-\Delta \tilde v= D\tilde u^q, \quad y\in  {\mathbb R}^N,
 \end{cases}
\end{align*}
and $\tilde u^{\frac{1}{\alpha}}(0)+\tilde v^{\frac{1}{\beta}}(0)
+|\nabla \tilde{u}(0)|^{\frac{1}{\alpha+1}}+|\nabla \tilde{v}(0)|^{\frac{1}{\beta+1}}=1$.
This contradicts the Liouville-type result of Lane-Emden system in \cite{Sou09}
and concludes the proof.
\end{proof}

In addition to Theorem \ref{th3}, we shall at the same time prove the following, corresponding gradient estimates,
which will be useful in the proof of Theorem \ref{cond2}.

\begin{proposition}\label{th3B} 
Under the assumptions of Theorem \ref{th3}, $(u,v)$ also satisfies the estimates
\begin{align}
|\nabla u(x)|\leq C|x|^{-\alpha -1-\frac{a+bp}{pq-1}}, 
\quad
|\nabla v(x)|\leq C|x|^{-\beta -1 -\frac{b+aq}{pq-1}},
\label{sing1B}
\end{align}
for $0<|x|<\rho/2$ (resp. $|x|>2\rho$).
\end{proposition}

{\it Proof of Theorem \ref{th3} and Proposition \ref{th3B}.} 
Assume either $\Omega=\{x\in{\mathbb R}^N;\, 0<|x|<\rho\}$ and $0<|x_0|<\rho/2$,
or $\Omega=\{x\in{\mathbb R}^N;\, |x|>\rho\}$ and $|x_0|>2\rho$.
Let $R_0=|x_0|/2> 0$.
We rescale $(u,v)$ by setting
\begin{align*}
U(y)=R_0^{\alpha+\frac{a+bp}{pq-1}}u(x_0+R_0y), V(y)=R_0^{\beta+\frac{b+aq}{pq-1}}v(x_0+R_0y).
\end{align*}  
Then $(U,V)$ is solution of 
\begin{align*}
\begin{cases}
-\Delta U=c(y)V^p,y\in B(0,1)\\
-\Delta V=d(y)U^q, y\in B(0,1).
\end{cases}
\end{align*}
where 
\begin{align*}
c(y)=|y+ \frac{x_0}{R_0}|^a, \quad d(y)= |y+ \frac{x_0}{R_0}|^b.
\end{align*}
Notice that $|y+ \frac{x_0}{R_0}| \in [1,3]$, $\forall y\in \overline B(0,1)$. Moreover $\|c\|_{C^1(\overline B_1)}\leq C(a)$ and  $\|d\|_{C^1(\overline B_1)}\leq C(b)$, then applying Lemma \ref{lem3} we have 
$U(0)+V(0)+|\nabla U(0)|+|\nabla V(0)|\leq C$. Hence
$$u(x_0)\leq CR_0^{-\alpha-\frac{a+bp}{pq-1}},\qquad v(x_0)\leq CR_0^{-\beta-\frac{b+aq}{pq-1}}$$
and
$$|\nabla u(x_0)|\leq CR_0^{-\alpha-1-\frac{a+bp}{pq-1}},\qquad |\nabla v(x_0)|\leq CR_0^{-\beta-1-\frac{b+aq}{pq-1}}.$$
The announced results are proved.
\qed

\medskip

\subsection{A priori bound}
{\it Proof of Theorem \ref{th4}.}
Suppose that Theorem \ref{th4} is false. Let $d=\text{dist}(0,\partial\Omega)>0$. Due to 
estimate (\ref{sing1}) in Theorem~\ref{th3}, all solutions of (\ref{bound}) are uniformly bounded, away from $\{0\}\cup \partial \Omega$. Then there are only two following possibilities.

{\bf Case 1:} There exists sequence of solutions $(u_k, v_k)$ and a sequence of points $P_k\to P\in \partial \Omega$ such that 
\begin{align}\label{61a}
N_k=\sup_{x\in\Omega:\text{dist}(x,\partial\Omega)<d/2}\left(u^{\frac{1}{\alpha}}_k(x)
+v^{\frac{1}{\beta}}_k(x)\right)
=u^{\frac{1}{\alpha}}_k(P_k)
+v^{\frac{1}{\beta}}_k(P_k)\to \infty \text{ as } k\to \infty.
\end{align}
 We rescale solution according to
\begin{align*}
 U_k(y)=\lambda_k^\alpha U(P_k+\lambda_k y), \quad V_k(y)=\lambda_k^\beta V(P_k+\lambda_k y); \quad \lambda_k=N_k^{-1}
\end{align*}
then 
\begin{align*}
\begin{cases}
-\Delta U_k=|P_k+\lambda_k y|^aV_k^p,\\
-\Delta V_k=|P_k+ \lambda_ky|^bU_k^q.
\end{cases}
\end{align*}

By the argument similar to that in \cite{GS81b}, there exists $\ell_1, \ell_2>0$ and
 functions $U,V$ solving the following problem in the half-space
\begin{align*}
\begin{cases}
- \Delta U= \ell_1 V^p, \quad x\in H^N_s\\
- \Delta V= \ell_2 U^q, \quad x\in H^N_s\\
 U(x)=V(x)=0, \quad x\in \partial H^N_s\\
U^{\frac{1}{\alpha}}(0)+V^{\frac{1}{\beta}}(0)=1,
\end{cases}
\end{align*}
where $H^N_s:=\{y\in {\mathbb R}^N: y_1>-s\}$ for some $s>0$. 
In view of assumption (\ref{hyper}), this contradicts the Liouville-type result 
of \cite[Theorem~4.2]{PQS07} for  the Lane-Emden system in a half-space.

\smallskip

{\bf Case 2:} There exists a sequence of solutions $(u_k, v_k)$ and a sequence of points $P_k\to 0\in \Omega$ such that
\begin{align*}
M_k=\sup_{|x|<d/2}\left(u^{\frac{1}{\alpha+\frac{a+bp}{pq-1}}}_k(x)
+v^{\frac{1}{\beta+\frac{b+aq}{pq-1}}}_k(x)\right)
=u^{\frac{1}{\alpha+\frac{a+bp}{pq-1}}}_k(P_k)
+v^{\frac{1}{\beta+\frac{b+aq}{pq-1}}}_k(P_k)\to \infty \text{ as } k\to \infty.
\end{align*}
We denote by
\begin{align*}
U_k(y)=\lambda_k^{\alpha+\frac{a+bp}{pq-1}}u_k(P_k+\lambda_ky), V_k(y)=\lambda_k^{\beta+\frac{b+aq}{pq-1}}v_k(P_k+\lambda_ky),\quad \lambda_k=M_k^{-1}.
\end{align*}
Then $(U_k,V_k)$ is solution to 
\begin{align}\label{62}
\begin{cases}
-\Delta U_k=|y+ \frac{P_k}{\lambda_k}|^aV_k^p,y\in B(0,\frac{d}{2\lambda_k})\\
\noalign{\vskip 1mm}
-\Delta V_k=|y+\frac{P_k}{\lambda_k}|^bU_k^q, y\in B(0,\frac{d}{2\lambda_k}).
\end{cases}
\end{align}
Moreover, it follows from estimate (\ref{sing1}) in Theorem \ref{th3}. that the sequence $\lambda_k^{-1}|P_k|=|P_k|M_k$ is bounded.
We may thus assume that $\lambda_k^{-1} P_k \to x_0$ as $k\to \infty$.

From (\ref{62}), by using the elliptic estimates and standard imbeddings, we deduce that some subsequence of $(U_k, V_k)$ converges in $C_{\text{loc}}({\mathbb R}^N)$ to a solution $(U,V)$ in ${\mathbb R}^N$  of the following system
\begin{align*}
\begin{cases}
-\Delta U=|y+ x_0|^aV^p,\ \ y\in {\mathbb R}^N\\
\noalign{\vskip 1mm}
-\Delta V=|y+x_0|^bU^q,\ \  y\in {\mathbb R}^N .
\end{cases}
\end{align*}
with 
$$U^ {\frac{1}{\alpha+\frac{a+bp}{pq-1}}}(0)+V^ {\frac{1}{\beta+\frac{b+aq}{pq-1}}}(0)=1.$$
After a space shift, this gives a contradiction with Theorem~\ref{cond2}.
\qed

\medskip

\section{Proof of Theorem \ref{cond2}.}
Let $(u,v)$ be a positive solution of system (\ref{1}). 
By the Rellich-Pohozaev identity (Lemma~\ref{pohozaev})  with (\ref{choicea1a2}), we have
\begin{align}
\int_{B_R}|x|^av^{p+1}\, dx
&+\int_{B_R}|x|^bu^{q+1}\, dx  \notag\\
&\leq H(R):=CR^{N+a}\int\limits_{S^{N-1}}v^{p+1}(R,\theta)\,d\theta 
+CR^{N+b}\int\limits_{S^{N-1}}u^{q+1}(R,\theta)\,d\theta  \notag\\
&+CR^N\int\limits_{S^{N-1}}\left(|D_xu(R,\theta)|+R^{-1}u(R,\theta) \right)
\left(|D_xv(R,\theta)|+R^{-1}v(R,\theta) \right)\,d\theta.  \notag
\end{align}
Now, by Theorem~\ref{th3} and Proposition \ref{th3B}, for $x\neq 0$, we have
$$u(x)\leq C|x|^{-\alpha-\frac{a+bp}{pq-1}},\qquad v(x)\leq C|x|^{-\beta-\frac{b+aq}{pq-1}}$$
and
$$|\nabla u(x)|\leq C|x|^{-\alpha-1-\frac{a+bp}{pq-1}},\qquad |\nabla v(x)|\leq C|x|^{-\beta-1-\frac{b+aq}{pq-1}}.$$
By straightforward calculations, it follows that 
$$H(R)\leq CR^{N-2-(1+\frac{b}{2})\alpha-(1+\frac{a}{2})\beta} \to 0, \text{ as } R\to \infty, $$ 
due to (\ref{1a}) (which is equivalent to (\ref{la1})). Therefore, $u\equiv v\equiv 0$.
\qed

\medskip

\section{Appendix}

We start with the following simple Lemma.

\begin{lemma}\label{lemDistrib}
Let $a,b>-2$, $pq>1$, $N\geq 3$, $0\in \Omega$ and $(u,v)$ be positive solution of (\ref{1}).
Then:
\begin{equation}\label{estimGrad}
\hbox{ There exists a sequence $\eps=\eps_i\to 0$ such that }
\int\limits_{|x|=\eps_i} \bigl( |\nabla u|^2+|\nabla v|^2\bigr)  \,d\sigma_{\eps_i}\to 0.
\end{equation}
Moreover, $(u,v)$ is a distributional solution of (\ref{1}).
\end{lemma}

\begin{proof}
If $a,b\geq 0$, the result is immediate. Let us first consider  $a,b\in (-2,0)$. We note that $u,v\in W^{2,k}_{loc}({\mathbb R}^N)$ with $1<k\leq N/2$,  due to $|a|, |b|<2$ and elliptic regularity.
By Sobolev imbedding, it follows that 
\begin{align}\label{7a}
|\nabla u|, |\nabla u| \in L^N_{loc}({\mathbb R}^N).
 \end{align}
By the same arguments, (\ref{7a}) still holds if $a\geq0$ or $b\geq 0$.
Consequently,
\begin{align*}
\int\limits_{\rho=0}^\eps\int\limits_{|x|=\rho} \left(|\nabla u|^2+|\nabla v|^2\right) \,d\sigma_\eps\,d\rho
&=\int\limits_{|x|<\eps}\left( |\nabla u|^2+|\nabla v|^2 \right)\,dx\\
&\leq C\eps^{N-2}\left(\|\nabla u\|^2_{L^N(B_\eps)}+\|\nabla u\|^2_{L^N(B_\eps)}\right),
\end{align*}
and assertion (\ref{estimGrad}) follows.

Let now $\varphi\in C^\infty_0(\Omega)$ and denote $\Omega_\eps=\Omega\cap\{|x|>\eps\}$ for $\eps>0$ small.
From (\ref{1}), using Green's formula, we obtain
\begin{align*}
\Bigl|\int\limits_{\Omega_\eps}  \left(|x|^av^p\varphi+ u\Delta\varphi\right)dx\Bigr|
=\Bigl|-\int\limits_{\Omega_\eps} \varphi\Delta u\,dx+\int\limits_{\Omega_\eps} u\Delta\varphi\,dx\Bigr|
=\Bigl|\int\limits_{|x|=\eps} \varphi\,\frac{\partial u}{\partial r} \,d\sigma_\eps-\int\limits_{|x|=\eps} u\,\frac{\partial \varphi}{\partial r} \,d\sigma_\eps\Bigr|.
\end{align*}
Similarly,
\begin{align*}
&\Bigl|\int\limits_{\Omega_\eps} |x|^bu^q\varphi\,dx+\int\limits_{\Omega_\eps} v\Delta\varphi\,dx\Bigr|
=\Bigl|\int\limits_{|x|=\eps} \varphi\,\frac{\partial v}{\partial r} \,d\sigma_\eps-\int\limits_{|x|=\eps} v\,\frac{\partial \varphi}{\partial r} \,d\sigma_\eps\Bigr|.
\end{align*}
Passing to the limit with $\eps=\eps_i$, we conclude that $(u,v)$ is a distributional solution of (\ref{1}).
 \end{proof}
\medskip
{\it Proof of Lemma \ref{pohozaev}.}
Since $u$ is a solution of (\ref{1}) then
\begin{align}\notag
&(x.\nabla u)\Delta v+(x.\nabla v)\Delta u=-(x.\nabla u)|x|^bu^q-(x.\nabla v)|x|^av^p\\
&=-\text{div}\left(x|x|^b\frac{u^{q+1}}{q+1}+x|x|^a\frac{v^{p+1}}{p+1}\right)+\frac{N+b}{q+1}|x|^bu^{q+1}+\frac{N+a}{p+1}|x|^av^{p+1}.\label{a11}
\end{align}
Integrating (\ref{a11})  on $B_R\setminus B_\varepsilon$ and letting $\varepsilon \to 0$, we have
\begin{align}\notag
\int\limits_{B_R}(x.\nabla u)\Delta v+(x.\nabla v)\Delta u\,dx=\int\limits_{B_R}\left(\frac{N+b}{q+1}|x|^bu^{q+1}+\frac{N+a}{p+1}|x|^av^{p+1}\right)dx\\
-\int\limits_{|x|=R}\left(R^{1+b}\frac{u^{q+1}}{q+1}+R^{1+a}\frac{v^{p+1}}{p+1}\right) d\sigma_R.\label{2b}
\end{align}

On the other hand, we have 
\begin{align}\label{2c}
\int\limits_{B_R\setminus B_{\eps}}\nabla u.\nabla v\,dx&=-\int\limits_{B_R\setminus B_{\eps}}u\Delta v \,dx 
+\int\limits_{|x|=R}uv'\,d\sigma_R-\int\limits_{|x|=\eps}uv'\,d\sigma_\eps\notag\\
&=\int\limits_{B_R\setminus B_{\eps}}|x|^bu^{q+1}\,dx+\int\limits_{|x|=R}uv'\,d\sigma_R-\int\limits_{|x|=\eps}uv'\,d\sigma_\eps.
\end{align}
Letting $\eps=\eps_i\to 0$ in (\ref{2c}), where $\eps_i$ is given by Lemma~\ref{lemDistrib}, we obtain
\begin{align*}
\int\limits_{B_R}\nabla u.\nabla v\,dx=\int\limits_{B_R}|x|^bu^{q+1} \,dx +\int\limits_{|x|=R}uv'\,d\sigma_R.
\end{align*}
Similarly,
\begin{align*}
\int\limits_{B_R}\nabla u.\nabla v\,dx=\int\limits_{B_R}|x|^av^{p+1} \,dx +\int\limits_{|x|=R}u'v\,d\sigma_R.
\end{align*}
Hence, for $a_1+a_2=N-2$, we have 
\begin{align}\label{2d}
\int\limits_{B_R}(N-2)\nabla u.\nabla v\,dx=\int\limits_{B_R}\left(a_1|x|^av^{p+1}+a_2|x|^bu^{q+1}\right) \,dx +\int\limits_{|x|=R}(a_1u'v+a_2uv')\,d\sigma_R.
\end{align}
By direct 
computation, we have the following identity
\begin{align}\label{3a}
(x.\nabla u)\Delta v+(x.\nabla v)\Delta u-(N-2)\nabla u.\nabla v=\text{div}\left[(x.\nabla u)\nabla v+(x.\nabla v)\nabla u-x\nabla u.\nabla v\right].
\end{align}
Integrating (\ref{3a}) on $B_R\setminus B_\varepsilon$ and letting $\eps=\eps_i\to 0$, where $\eps_i$ is given by Lemma~\ref{lemDistrib}, we have
\begin{align}\label{3b}
 \int\limits_{B_R}\left[(x.\nabla u)\Delta v+(x.\nabla v)\Delta u-(N-2)\nabla u.\nabla v\right] dx=\int\limits_{|x|=R} R\left(2u'v'-\nabla u.\nabla v\right) d\sigma_R
\end{align}
The Rellich-Pohozaev identity follows from (\ref{2b}),  (\ref{2d}),  and (\ref{3b})
\qed

\medskip
For the proof of Lemma \ref{lem4a}, we need the following lemma (see  \cite[Lemma 3.2]{BC98} and \cite{AS11}).
\begin{lemma}\label{lem6a}
Assume $h\in L^\infty(B_3\setminus B_{1/2})$ is nonnegative, and $u\geq 0$ satisfies
$$-\Delta u\geq h(x)\quad \text{ in } B_3\setminus B_{1/2}.$$
There exists a constant $C=C(N)$ such that
\begin{align}
 \inf_{B_2\setminus B_{1}}u\geq C\int_{B_2\setminus B_1}h(x)dx.
\end{align}
 \end{lemma}

{\it Proof of Lemma \ref{lem4a}.}
Let $m_1(R)=\inf_{B_{2R}\setminus B_R} u$, $m_2(R)=\inf_{B_{2R}\setminus B_R}v$. 
It follows from Lemma \ref{lem6a} that
 \begin{align}
  m_1(R)\geq CR^{2-N}\int\limits_{B_{2R}\setminus B_R}|x|^av^pdx\geq CR^{2+a}m^p_2(R),\quad R>\rho,\label{si2}\\
m_2(R)\geq CR^{2-N}\int\limits_{B_{2R}\setminus B_R}|x|^bu^qdx\geq CR^{2+b}m^q_1(R), \quad R>\rho.\label{si3}
 \end{align}
Therefore,
 \begin{align}\label{si1}
  m_1(R)\geq C R^{2+a+p(2+b)}m_1 ^{pq}(R),\quad m_2(R)\geq C R^{2+b+q(2+a)}m_2^{pq}(R), \quad R>\rho,
 \end{align}
%
hence
$$m_1(R)\leq CR^{-\alpha-\frac{a+bp}{pq-1}}, \quad m_2(R)\leq CR^{-\beta-\frac{b+aq}{pq-1}}.$$
Combining this with (\ref{si2}) and (\ref{si3}), we have the desired estimates in Lemma \ref{lem4a}.
\qed

\bibliographystyle{plain}

\end{document}